\documentclass{amsart}
\usepackage{amssymb}
\usepackage{graphics,color}
\usepackage[latin1]{inputenc}
\usepackage[pagebackref=false]{hyperref}

 \usepackage{latexsym}
 \usepackage{graphics}
 \usepackage{color}

 \newcommand{\m}{\mathfrak{m} }

 \newcommand{\p}{\operatorname{P}}

 \newcommand{\indeg}{\operatorname{indeg}}
 \newcommand{\gr }{\operatorname{gr}}
 \newcommand{\ann}{\operatorname{ann}}
 \newcommand{\lin}{\operatorname{lin}}
  \newcommand{\reg}{\operatorname{reg}}

 \newcommand{\Tor}{\operatorname{Tor}}

\theoremstyle{plain}

\newtheorem{thm}{Theorem}[section]

\newtheorem{cor}[thm]{Corollary}
\newtheorem{prop}[thm]{Proposition}
\newtheorem{rem}[thm]{Remark}

\newtheorem{defin-rem}[thm]{Definition and Remark}

\newtheorem{defen}[thm]{Definition}

\newtheorem{con}[thm]{Construction}

\title[ ] {on the regularity and Koszulness of modules over local rings}

\author{Rasoul Ahangari Maleki}

\address{Faculty of Mathematical Sciences and Computer, Kharazmi  University, Tehran}\email{rasoulahangari@gmail.com\\
 std\_ahangari@khu.ac.ir }
\numberwithin{equation}{thm}

\begin{document}
\maketitle
\begin{abstract}
Koszul modules  over Noetherian local rings $R$ were introduced by Herzog and Iyengar and   they possess  good homological properties, for instance their Poincare' series is rational.  It is an interesting problem to characterize classes of Koszul modules.  Following the idea traced by Avramov, Iyengar and Sega,  we take advantage of the existence of special filtration  on $R$ for proving that large classes of  $R$-modules over Koszul rings are Koszul modules. By using this tool we reprove and extend some results obtained by Fitzgerald.
\end{abstract}
\thispagestyle{empty}

\section*{Introduction}
Consider a Noetherian local ring (or a standard graded $k$-algebra) $R$ with maximal (or maximal homogeneous) ideal $\mathfrak{m}$ and residue field $k$. Given a finitely generated $R$-module
$M$ and $\mathbf{F}$ its minimal  free resolution, denote  by $\lin^{R}(\mathbf{F})$ the associated graded complex $\mathbf{F}$
arising from the standard $\mathfrak{m}$-adic filtration; the details are given in Section $1$. This construction has been investigated by Herzog and Iyengar (\cite{H-i}).\\
 \indent Following \cite{H-i}, a finitely generated $R$-module $M$ is Koszul if $\textsc{H}_{\emph{i}}(\lin^{R}(\mathbf{F}))=0$
for all $i>0$. Such modules are characterized by the property that their associated graded module $\gr_{\mathfrak{m}}(M)$ has a linear resolution
over the associated graded ring $\gr_{\m}(R)$.
The ring $R$ is Koszul if $k$ is a Koszul module.\\
\indent From certain point of views, Koszul algebras behave homologically as polynomial rings.
For instance, if $R$ is a graded Koszul algebra, then  every finitely generated graded module over $R$ has regularity bounded by its regularity over a polynomial ring of which the Koszul algebra is a homomorphic image.\\
\indent There are some techniques to prove that a graded algebra $R$ is Koszul. One of them is the existence of a special filtration
called Koszul filtration  which is introduced by Conca, Trung and Valla  (\cite{CTV}). Roughly speaking, a Koszul filtration
 for an algebra $R$ is a family of ideals of $R$ generated by linear forms such that any ideal of the family can be filtered in such a way that all the successive colon ideals belong to the family. If an algebra has a Koszul filtration, then it is necessarily Koszul (\cite[Theorem 1.2]{CTV}).
  Conca, Valla and Rossi in \cite{CRV} defined a special Koszul filtration, called  Gr\"{o}bner flag, which is just a Koszul filtration supported on a single complete flag of linear forms.\\
 \indent The goal of this paper is to take advantage of the existence of Koszul filtration and Gr\"{o}bner flag
  to determine some classes of Koszul modules over Koszul rings. In Section $2$, we will be concerned with three classes of rings:

 \begin{enumerate}
 \item[$(i)$]a local ring $(R,\m)$ which has an element $0\neq x\in \m$ with $x^{2}=0$ and $x\m=\m^{2}$;
 \item[$(ii)$] the class of Cohen-Macaulay local rings with minimal multiplicity;
 \item[$(iii)$] a local ring $(R,\m)$ with the property that $\ann(x)\m=\m^{2}$ for all $x\in \m\setminus \m^{2}$.
 \end{enumerate}
   We investigate the Koszul property and we find some upper bounds for the regularity of modules over these classes of local rings.\\
  \indent Theorem \ref{reg} states that if $0\neq x\in \mathfrak{m}$ satisfies $x^{2}=0$ and $x\mathfrak{m}=\mathfrak{m}^{2}$ (such an element is called  a Conca generator of $\mathfrak{m}$), then modules  annihilated by $x$ are Koszul. This result was already proved in
  \cite[Theorem 3.2]{AIS}, here we give an alternative proof as an application of the existence of Gr\"{o}bner flag. \\
  \indent It is known that a Cohen-Macaulay local ring $R$ of minimal multiplicity is Koszul (see  \cite[Theorem 5.2]{H-V-W}). We show, in Theorem \ref{min-mult}, that every finitely generated  $R$-module annihilated by a minimal reduction of $\mathfrak{m}$, is Koszul.\\
  \indent Let $(R,\mathfrak{m})$ be a local ring and $\ann(x)\mathfrak{m}=\mathfrak{m}^{2}$
  for every $x\in \mathfrak{m}\setminus \mathfrak{m}^{2}$. Then, by a result of Fitzgerald (\cite[Theorem 3.6]{F}), the ring $R$ is a Koszul ring. Here, we extend and improve this result. More precisely, we show that the associated graded ring of $R$ has a Koszul filtration. This gives us an alternative proof for the Koszulness of $R$. Moreover, for a finitely generated $R$-module $M$ we establish the following:

 \begin{enumerate}
 \item[$(i)$] $M$ is Koszul if there exists $x\in \m\setminus \m^{2}$ such that $\ann(x)M$=0;
 \item[$(ii)$] if $x_{1},\cdots,x_{s}$ are in $\m\setminus \m^{2}$, then $(x_{1},\cdots,x_{s})M$ is Koszul;
 \item[$(iii)$] $\reg _{R}(M)\leq 1 $ .
 \end{enumerate}

\smallskip

\indent Throughout  the paper, all rings are commutative Noetherian,
all modules are finitely generated
and $k$ denotes a field. Also, we will use
$\mathbb{Z}$ (respectively $\mathbb{N}_{0}$) to denote
the set of integers (respectively non-negative integers).

\smallskip
\section{\bf Notations and Generalities}

\indent In this section, we are going to fix some notations and generalities.\\
\indent  Let $R=\bigoplus_{i\in \mathbb{N}_{0}}R_{i}$
 be a standard graded $k$-algebra, that is  $R_{0}=k$ and $R$ is generated (as a $k$-algebra) by elements of
degree $1$.\\

\begin{rem}\label{nota}
 Let $M=\bigoplus_{i\in \mathbb{Z}} M_{i}$  be a graded $R$-module.
 \begin{enumerate}
 \item[$(1)$] For each $d\in\mathbb{Z}$ we denote by
$M(d)$ the graded $R$-module with $M(d)_{p}=M_{d+p}$, for all $p\in\mathbb{Z}.$
Denote by $\mathfrak{m}$ the maximal homogeneous ideal of $R$.
Then we may consider $k$ as a graded $R$-module via the identification $k=R/\m$.\\
\item[$(2)$] The Hilbert series of $M$ is defined by
$\textsc{H}_{M}(z)=\sum_{i\in\mathbb{Z}} \dim_{k}(M_{i})z^{i}\in \mathbf{Q}[|z|][z^{-1}]$.
\item[$(3)$] A minimal graded free resolution of $M$ as an $R$-module is a
complex of free $R$-modules
\[\mathbf{F}= \cdots F_{i}\xrightarrow{\partial_{i}}F_{i-1}\rightarrow\cdots\rightarrow F_{1}\xrightarrow{\partial_{1}} F_{0}\rightarrow 0\]\\
such that $\textsc{H}_{i}(\mathbf{F})=0$ for $i>0$, $\textsc{H}_{0}(\mathbf{F})=M$ and
$\partial_{i}(F_{i})\subseteq \mathfrak{m}F_{i-1}$ for all $i\in \mathbb{N}_{0}$. Each $F_{i}$ is isomorphic to
a direct sum of copies of $R(-j)$, for $j\in \mathbb{Z}$. Such a resolution exists and
any two minimal graded free resolutions are isomorphic as complexes of graded $R$-modules.
So, for all $j\in\mathbb{Z}$ and $i\in\mathbb{N}_{0}$  the number of direct summands of $F_{i}$ isomorphic to $R(-j)$  is an invariant of $M$,
 called the $ij$-th graded Betti number of $M$ and denoted by $\beta_{i j}^{R}(M)$. Also, by definition,
 the $i$-th Betti number of $M$ as an $R$-module, denoted by  $\beta^{R}_{i}(M)$, is the rank of $F_{i}$.
 By construction, one has $\beta_{i }^{R}(M)=\dim_{k}\Tor^{R}_{i}(M,k)$ and
$\beta_{i j}^{R}(M)=\dim_{k}\Tor^{R}_{i}(M,k)_{j}.$
 \item[$(4)$] The Poincar\'{e} series of $M$ is defined by
\[\p_{M}^{R}(t)=\sum_{i}\beta_{i}^{R}(M)t^{i}\in \mathbf{Q}[|t|],\]
and its bigraded version is

\[\p^R_{M}(s,t)=\sum_{i,j}^{R}\beta_{i j}(M)t^{i}s^{j}\in\mathbf{Q}[s][|t|].\]

\item[$(5)$]  The Castelnuovo-Mumford regularity
of $M$ is defined by \\
\[\reg_{R}(M)=\sup\{j-i| j\in \mathbb{Z}, i\in\mathbb{N} \ \text{and}\ \beta_{i j}^{R}(M)\neq 0 \}.\]\\
This invariant is, after Krull dimension and multiplicity, perhaps the most important invariant of
 a finitely generated graded $R$-module.\\
\item[$(6)$]  Let $R$ be a standard graded $k$-algebra and  $M$ be a graded $R$-module.
The initial degree $\indeg (M)$ of $M$ is the minimum of the $i$ such that $M_{i}\neq 0$.
We say that $M$ has  $d$-linear resolution if $d=\indeg(M)$ and $\reg_{R}(M)=d$. In the case $\indeg(M)=0,$ one can see that
$\reg_{R}(M)=0$ if and only if  the graded Poincar\'{e}  series $\p^R_{M}(s,t)$ can be written
 as a formal power series in the product $st$.
\end{enumerate}
\end{rem}
  We define  now  the notion of regularity to local rings, adopting the notion of
(\c{S}ega [3, Section 2]).
\begin{defin-rem} Let  $(R,\mathfrak{m},k)$ be a local ring (or a standard graded $k$-algebra
with the maximal homogeneous ideal $\m$). Let $M$ be an $R$-module (or a graded $R$-module).
\begin{enumerate}
\item[$(1)$] We denote by $R^{g}$
the associated graded ring $\bigoplus_{i\geq 0}\mathfrak{m}^{i}/\mathfrak{m}^{i+1}$, which is a standard
graded $k$-algebra, and by $M^{g}$ the associated graded $R^{g}$-module
$\bigoplus_{i\geq 0}\mathfrak{m}^{i}M/\mathfrak{m}^{i+1}M$ with respect to the $\mathfrak{m}$-adic filtration. \\
Notice that in the graded case since
$R$ is standard graded, the $k$-algebras $R$  and $R^{g}$ are naturally
 isomorphic. But $M$ and $M^{g}$ need not be isomorphic unless $M$ is generated by elements of degree zero.
\item[$(2)$] Let  $(R,\mathfrak{m},k)$ be a local ring and $N$ be an $R$-module. The regularity of $N$ over $R$ is defined by
\[\reg_{R}(N):=\reg_{R^{g}}(N^{g}).\]
\end{enumerate}
\end{defin-rem}

 \smallskip

 \section{Koszul modules }
We start this section by the construction of the {\em linear part} of
a complex over a local ring $(R,\m,k)$, recalled below.
\begin{con}
A complex of $R$-modules
\[\mathbf{C}=\cdots \rightarrow C_{n+1}\rightarrow C_{n}\xrightarrow{\partial_{n}} C_{n-1}\rightarrow\cdots\]
is said to be minimal if $\partial_{n}(C_{n})\subseteq \m C_{n-1}$ for all $n\in\mathbb{Z}$. The standard filtration
$\mathcal{F}$ of a minimal complex $\mathbf{C}$ is defined by subcomplexes $\{\mathcal{F}^{i}\mathbf{C}\}_{i\geq 0},$
where
\[(\mathcal{F}^{i}\mathbf{C})_{n}=\m^{i-n}C_{n}\ \text{for all} \ n\in\mathbb{Z}\]
with the convention that $\m^{j}=0$ for $j\leq 0$.
The minimality of the complex $\mathbf{C}$ ensures that
$\partial(\mathcal{F}^{i}\mathbf{C})\subseteq \mathcal{F}^{i}\mathbf{C}$
for each integer $i$ so that $\mathcal{F}$ is a filtration of $\mathbf{C}$.
The associated graded
complex with respect to this filtration is denoted by $\lin^{R}(\mathbf{C})$, and called the linear part of $\mathbf{C}$.
By construction, $\lin^{R}(\mathbf{C})$ is a minimal complex of graded modules over the graded ring $R^{g}$ and
has the property that
$\lin^{R}(\mathbf{C})_{n}=C^{g}_{n}(-n).$
\end{con}
Recall that over local rings, each finitely
generated module has a minimal free resolution, and that this is unique (up to isomorphism).
Thus, one may speak of {\em the} minimal free resolution of such a module.

\indent Now, we introduce the main objects of interest,
Koszul rings and modules. We refer to \cite{H-i} and \cite{S} for detailed studies of Koszul property.\\

\begin{defen}
Let $(R,\m,k)$ be a local ring. An $R$-module $M$ is said to be Koszul if $\lin^{R}(\mathbf{F})$ is acyclic,
 where $\mathbf{F}$ is the minimal free resolution of $M$. $R$ is Koszul if its residue field $k$ is Koszul as a
 module over $R$.
\end{defen}
\begin{rem} The notion of Koszul modules can be defined in the graded case.
\begin{enumerate}
\item[$(1)$]Let $R$ be a standard graded $k$-algebra with the maximal homogeneous ideal $\mathfrak{m}$
and $M$ be a graded $R$-module.
  Let $\mathbf{F}$ be the  minimal graded free resolution
 of  $M$. As in the local case, one can define the standard filtration of  $\mathbf{F}$ with respect to $\mathfrak{m}$
 and the associated graded complex $\lin^{R}(\mathbf{F})$. We remind that since
$R$ is standard graded, the $k$-algebras $R$  and $R^{g}$ are naturally
 isomorphic. Given these considerations, one can speak of Koszul modules.
 If a graded $R$-module $M$ has a linear resolution, then it is Koszul, but the converse
 fails in general (see \cite[1.9]{H-i}).\\
 In keeping with tradition, when the $R$-module $k$ is Koszul, we say that $R$ is a Koszul algebra, rather than a Koszul ring.
In fact,  a local ring $R$ is a Koszul ring if and only if $R^{g}$ is a Koszul algebra.
\item[$(2)$] In both local and graded case, Koszul modules are characterized by the property that their associated graded module $M^{g}$ has  $0$-linear resolution over the associated graded ring $R^{g}$.\\
\end{enumerate}
\end{rem}
\indent S. Iyengar and T. R\"{o}mer in \cite{T-I} defined $R$ to be absolutely Koszul if every finitely generated $R$-module has a high syzygy module
 which is  Koszul. While absolutely Koszul rings have
to be Koszul, the converse does not hold; see the discussion in the introduction of \cite{H-i}.\\
   \indent The most important property of Koszul algebras is that every graded module over  a Koszul algebra has  finite regularity.
 Avramov and Eisenbud in \cite{A-E} proved that, if $R$ is a Koszul algebra which is a homomorphic image of a polynomial ring $S$,
then for every  graded $R$-module $M$\\
 \[\reg_{R}(M)\leq \reg_{S}(M).\]

 In the this section we will improve this upper bound for some classes  of Koszul rings.\\
\indent Using a result of \c{S}ega (\cite[ Proposition 2.3]{S}) and Herzog and Iyengar  (\cite[Proposition 1.5]{H-i}), we have the following
 equivalent conditions for the Koszul property of a module over a local ring or
 a standard graded algebra.
 \begin{thm}\label{Koszul-eq}
 Let $R$ be a local ring  (or a standard graded $k$-algebra) and $M$ be an $R$-module (or a graded $R$-module).
 The following statements are equivalent.
 \begin{enumerate}
 \item $M$ is a Koszul $R$-module.
  \item $M^{g}$ has $0$-linear resolution  .
 \item $\reg_{R^{g}}(M^{g})=0$.
 \item $\p^{R^{g}}_{M^{g}}(s,t)$ can be written as a formal power series in the product $st$.
 \end{enumerate}

 \end{thm}
 \indent By \cite[Proposition. 1.8]{H-i},   a Koszul module $M$ over a
local ring (or  a standard graded $k$-algebra) $R$ satisfies\\
\[\p^{R}_{M}(t)=\p^{R^{g}}_{M^{g}}(t)=\textsc{H}_{M^{g}}(-t)/\textsc{H}_{R}(-t).\]\\
When $R$ is a standard graded $k$-algebra, L\"{o}fwall  proved that this formula, with $M$ equal to $k$, is actually equivalent to $R$
being Koszul (see \cite[Theorem 1.2]{Lo}). In the following we improve this result to modules.

\begin{thm}
Let $R$ be a standard graded $k$-algebra   and  $M$ be
a  graded $R$-module. The following conditions are equivalent.

\begin{enumerate}
\item $M$ is a Koszul module.
 \item $\p^{R}_{M^{g}}(t)=\textsc{H}_{M^{g}}(-t)/ \textsc{H}_{R}(-t)$.
\end{enumerate}
\end{thm}
\begin{proof}
The implication $(1)\Rightarrow(2)$ follows from \cite[proposition 1.8]{H-i}. \\
We prove now $(2)\Rightarrow (1)$. Let $\p_{M}^{R}(t)=\sum_{i\geq0}b_{i}t^{i}$, $\textsc{H}_{R}(t)=\sum_{i\geq0}h_{i}t^{i}$ and $\textsc{H}_{M^{g}}(t)=\sum_{i\geq0}a_{i}t^{i}$. Consider the minimal
graded $R$-free resolution of $M^{g}$:

\[\cdots\rightarrow \bigoplus _{j}R(-j)^{\beta_{i ,j}}\rightarrow \cdots\rightarrow\bigoplus _{j}R(-j)^{\beta_{1, j}}\rightarrow \bigoplus _{j}R(-j)^{\beta_{0, j}}\rightarrow M^{g}\rightarrow 0.\]

We use induction on $i$ to show that $\beta_{i,j}=0$ for $i\geq 0$ with $i\neq j. $  The case $i=0$ is clear, because $M^{g}$
is generated by elements of degree $0$.
Let $i> 0$ and suppose by induction that $\beta_{r,j}=0$ for $r<i$ and $r\neq  j$. Then $b_{r}=\beta_{r,r}$ for $r<i$.
 The above resolution induces the following exact sequence of $k$-vector
spaces\\
\[0\rightarrow R_{0}^{\beta_{i,i}}\rightarrow R_{1}^{b_{i-1}}\rightarrow\cdots\rightarrow R_{i}^{b_{0}}\rightarrow (M^{g})_{i}\rightarrow 0.\]\\

\noindent Notice  that $R(-i-1)_{i}=0$.  Therefore  $a_{i} =\sum_{r=0}^{ i}(-1)^{r}\dim_{k} R_{i-r}^{\beta_{r,r}}$ and by assumption we have
$a_{i}=\sum_{r=0}^{r=i}(-1)^{r}b_{r}h_{i-r}. $ So, $\beta_{i,i}=b_{i}$, which implies that $\beta_{i,j}=0$ for $j>i $.
 Now, the result follows from  Theorem \ref{Koszul-eq}.
\end{proof}

\indent There are various methods for proving that a $k$-algebra is Koszul, depending on the kind of information and presentation of the algebra one has, see for example Section 3 in \cite{CNR}.
 One of  them is the concept of Koszul filtration. This notion, introduced in
 \cite{CTV}, was inspired by the work of Herzog, Hibi and Restuccia (\cite{H-H-R}) on strongly Koszul algebras.
 Let us first recall the standard definition of Koszul filtration.

 \begin{defen}\label{Koszul-filt}
 Let $R$ be a standard graded $k$-algebra. A family $\mathcal{F}$ of ideals of $R$
 is said to be a Koszul filtration of $R$ if:
 \begin{enumerate}
 \item[(1)]Every ideal $I\in \mathcal{F}$ is generated by linear forms (i.e homogeneous element of degree 1 ).
 \item [ (2)] The ideal (0) and the maximal homogeneous ideal $\m$ of $R$ belong to $\mathcal{F}$.
 \item [ (3)] For every $I\in \mathcal{F}$ different from (0), there exists $J\in \mathcal{F}$ such that $J\subseteq I
 , I/J  \ \text{is cyclic and}\ J:I\in \mathcal{F}. $
  \end{enumerate}
 \end{defen}
  An important class of rings with a Koszul filtration are rings defined by
 quadratic monomial relations.
 If $R = K[x_{1},\cdots, x_{n}]/I$ where $I$ is generated
 by monomials of degree $2$, then it is easy to see that the family of all the ideals
 generated by subsets of $\{x_{1},\cdots, x_{n}\}$ is a Koszul filtration for $R$.\\
 \indent The following result is well known  (see  \cite[Proposition 1.2]{CTV}).

 \begin{prop}\label{fit-lin}
 Let $\mathcal{F}$ be a Koszul filtration of $R$. Then for every ideal $I\in\mathcal{F},$
 the $R$-module $R/I$ has $0$-linear resolution.
 \end{prop}

A simple way to prove that an algebra is Koszul is to show that the defining ideal is generated by a  Gr\"{o}bner basis of quadrics.
 A ring has this  property  if it possesses  a special Koszul filtration, called  Gr\"{o}bner flag (see \cite[Theorem 2.4]{CRV}).

 \begin{defen}\label{flag}
 Let $R$ be a standard graded $k$-algebra with the maximal homogeneous ideal $\m$. A Gr\"{o}bner flag of $R$ is a
Koszul filtration $\mathcal{F}$ of $R$ which consists of a single complete flag. In other words,
a Gr\"{o}bner flag is a set of ideals $\mathcal{F} = \{(0)=(V_{0}), (V_{1}),\cdots, (V_{n-1}), (V_{n})=\m\},$
where $V_{i}$ is an $i$-dimensional subspace of $R_{1}$, $ V_{i}\subset  V_{i+1}$ and
there exists $j_{i}\in\{1,\cdots,n\}$ such that $(V_{i-1}):(V_{i})=(V_{j_{i}})$.\\
 \indent This is equivalent to say that there exists an ordered  system of generators $l_{1},\cdots,l_{n}$
of $R_{1}$ (a basis of the flag) such that for every $i=1,\cdots,n$ we have \\
\[(l_{1},\cdots,l_{i-1}):l_{i}=(l_{1},\cdots,l_{j_{i}}).\]
\end{defen}
\begin{defin-rem}
Let $R$ be a standard graded $k$-algebra with the   maximal homogeneous ideal $\m$. Assume that $R$ has a Koszul filtration, say $\mathcal{F}$.
  A chain $I_{0}\subset I_{1}\subset\cdots\subset I_{n}=\m$ of elements of $\mathcal{F}$ is called   a flag in $\mathcal{F}$ starting from $I_{0}$,
if  for all $1\leq i\leq n,$   $(I_{i-1}:I_{i})\in\mathcal{F}$ and $I_{i}/I_{i-1}$ is cyclic.
It is straightforward to see that, any Koszul filtration has a flag starting from $(0)$.\\
\indent If $\mathcal{F}$ is a Gr\"{o}bner flag, then for  every element
$I\in\mathcal{F}$ there exists  a flag in $\mathcal{F}$ starting from $I.$ But it is not a necessary condition. \\
\indent Let $S=k[x_{1},\cdots,x_{n}]$ be a polynomial ring and $R=S/I$ where $I$ is a graded ideal generated by monomials of degree $2$.
As we mentioned before the family of all the ideals generated by subsets of $\{x_{1},\cdots,x_{n}\}$ is a Koszul filtration of $R$. One can see that for each element $I$  of this filtration there exist a flag starting from $I$.
In \cite[Example 2.6]{CRV}, it is shown that this Koszul filtration of  $R=k[x,y,z]/(x^2, xy, yz, z^2)$ is not a Gr\"{o}bner flag.

\end{defin-rem}
For a flag $I_{0}\subset I_{1}\subset\cdots\subset I_{n}=\m$ in a Koszul filtration $\mathcal{F}$
 the following result gives  us a relation between the bigraded Poincar\'{e} series of modules over a component $R/I_{i}$ and $R$ itself, which will be used later.

 \begin{prop}\label{main}

 Let $R$ be  a standard graded $k$-algebra with the maximal homogeneous ideal $\m$, which has  a Koszul filtration
 $\mathcal{F}$. Assume  that  $I_{0}\subset I_{1}\subset\cdots\subset I_{n}=\m$ is a flag in $\mathcal{F}$.
  If $M$ is a graded $R/I_{r}$-module
 for some $0\leq r\leq n$, then \\
 \[\p^R_{M}(s,t) = \p^{R/I_{r}}_{M}(s,t) \p^R_{R/I_{r}}(s,t).\]

 \end{prop}

 \begin{proof}
 Let $1\leq i \leq n$. By definition of a flag,
 $I_{j_{i}}:=(I_{i-1}:I_{i})\in \mathcal{F}$ and $I_{i}/I_{i-1}$ is cyclic. So $I_{i}/I_{i-1}\cong(R/I_{j_{i}})(-1)$ and
  we get the exact sequence \\ \[0\longrightarrow( R/I_{j_{i}})(-1)\longrightarrow R/{I_{i-1}} \longrightarrow R/{I_{i}}\longrightarrow 0 \]\\
 of graded modules and homomorphisms. This induces the long exact sequence

 \[\cdots\longrightarrow \Tor_{t}^{R}(R/I_{j_{i}},k)_{s-1}\longrightarrow \Tor_{t}^{R}(R/I_{i-1},k)_{s} \longrightarrow \Tor_{t}^{R}(R/I_{i},k)_{s}\] \  \[ \longrightarrow \Tor_{t-1}^{R}(R/I_{j_{i}},k)_{s-1}\longrightarrow \cdots\]
 for all integers $s,t\geq 0$. As $R/I_{j_{i}}$ has $0$-linear resolution over $R$,
 $\Tor_{t}^{R}(R/I_{j_{i}},k)_{s-1}=0$ for $s= t$. Therefore, the map\\
 \[ \Tor_{t}^{R}(R/I_{i-1},k)_{t} \longrightarrow \Tor_{t}^{R}(R/I_{i},k)_{t}\] \\
 is injective for all $t\geq 0$. This, in conjunction with the fact that the graded $R$-module
 $\Tor^{R}_{t}(R/I_{i},k)$  is concentrated in degree
 $t$ for $i=1,\cdots,n$, implies that the map \\
 \[ \Tor_{t}^{R}(R/I_{i-1},k) \longrightarrow \Tor_{t}^{R}(R/I_{i},k),\] \\
 which is induced by the natural  epimorphism $R/I_{i-1}\longrightarrow R/I_{i},$
 is injective for all $1\leq i\leq n$. Hence for each $i=0,\cdots,n$,
 the homomorphism \\ \[  \Tor_{t}^{R}(R/I_{i},k) \longrightarrow \Tor_{t}^{R}(k,k),\] \\
 induced by the natural surjective ring homomorphism $  R/I_{i}\longrightarrow k,$ is injective. \\
 Now, it is not difficult  to see  that, the proof of the implication $(3)\Rightarrow(2)$ in \cite[Theorem 1.1]{L} translates smoothly to the graded case.
 Namely, that argument uses spectral sequences that preserve grading and it leads  to the following conclusion

 \begin{equation}\label{1}
 \p^R_{M}(s,t) = \p^{R/(I_{r})}_{M}(s,t) \p^R_{R/(I_{r})}(s,t).
 \end{equation}
  \end{proof}

 \begin{cor}\label{EQ-Koszul}

 Let the situations be as in \ref{main}. Then  $M$ is  Koszul  over $R$ if and only if $M$ is Koszul over $R/I_{r}$.
 \end{cor}

 \begin{proof}
 Using \ref{fit-lin} $R/I_{r}$ has  $0$-linear resolution over $R$. So, by Theorem \ref{Koszul-eq}, $\p^R_{R/I_{r}}(s,t)$
 can be written as a formal power series in the product $st$.
 Therefore, the equality  (\ref{1})  implies that $\p^{R}_{M^{g}}(s,t)$
 is a formal power series in the product $st$ if and only if $\p^{R/I_{r}}_{M^{g}}(s,t)$ has
 this property. Now, the result follows  from Theorem \ref{Koszul-eq}.
 \end{proof}

\begin{defin-rem}\label{Conca-g} Let $R$ be a ring and $I$ an ideal of $R$.
 We say that $x\in I$ is a Conca generator of $I $  if
 $x\neq 0=x^{2}$ and $xI=I^{2}$. One then has $I^{3} = (xI)I = x^{2}I = 0$, and also $x\not\in I^{2}$.
 The contrary would imply $x \in I^{2} = xI\subseteq I^{3} = 0$, a contradiction.\\
 \indent Let $R$ be   a standard graded $k$-algebra and let $l\in R_{1}$ be a Conca generator of
 the maximal homogeneous ideal $\m$ of $R. $ By \cite[Lemma 2.7]{CRV}, the algebra
 $R$ has a Gr\"{o}bner   flag $\mathcal{F}=\{ 0=(V_{0}), (V_{1}),\cdots, (V_{n-1}), (V_{n})=\m\} $
 with $(V_{1})=(l)$. Thus for a graded $R$-module  $M$ with $lM=0,$
 $M$ is Koszul over $R$ if and only if $M$ is Koszul over $R/lR$, by Corollary \ref{EQ-Koszul}.
\end{defin-rem}

 \indent Avramov, Iyengar and \c{S}ega proved that over a local ring whose maximal ideal has
 a Conca generator, every module has a Koszul syzygy module
 (see \cite{AIS}). In the following we give some information about
 the Koszul modules and the regularity of modules over these  kind of rings. The first statement  was already proved  in  \cite{AIS}, but we insert here a proof
   as a consequence of Corollary \ref{EQ-Koszul}.

 \begin{thm}\label{reg}
 Let $( R, \m, k)$ be a local ring, $x$ be a Conca generator of $\mathfrak{m}$ and $M$ be an $R$-module. Then  the following hold.
 \begin{enumerate}
 \item  $M$ is Koszul if $xM=0$;
 \item $\reg _{R}(M)\leq 1$;
 \item $\mathfrak{m}M$ is Koszul;
 \item For every ideal $I$ of $R$ the quotient ring $R/I$ is a Koszul ring.
 \end{enumerate}
\end{thm}

\begin{proof}
 $(1)$ Let $M$ be an $R$-module such that $xM=0$. Set $x^{\ast} $ be the corresponding initial form of $x$ in $R^{g}.$ Thus
  \[x^{\ast}\in R^{g}_{1}  ,\qquad {x^{\ast}}^2=0  ,\qquad x^{\ast}R^{g}_{1}=R^{g}_{2}  \qquad \text{and} \ \ x^{\ast}M^{g} = 0.\]\\
  Let $n=\dim_{k}R^{g}_{1}.$ Then, in view of Definition and Remark \ref{Conca-g},   $R^{g}$ has a  Gr\"{o}bner   flag
  $\mathcal{F}=\{ 0=(V_{0}), (V_{1}),\cdots, (V_{n-1}), (V_{n})=(R^{g}_{1})\} $,
  such that $(V_{1})=(x^{*})$. And
  $M^{g}$ is Koszul over $R^{g}$ if and only if $M^{g}$ is Koszul over $R^{g}/x^{\ast}R^{g}$. Considering the fact that  \[(R^{g}/x^{\ast}R^{g})_{i}= 0 \qquad \text{for all}\ \ i\geq 2,\] \\
 $M^{g}$ has $0$-linear resolution over $R^{g}/x^{\ast}R^{g}$.  Therefore, $M^{g}$ is Koszul
  as an $R^{g}$-module by Corollary \ref{EQ-Koszul}.

 (2) The exact sequence \\
 \[ 0\longrightarrow x^{\ast}M^{g} \longrightarrow M^{g} \longrightarrow M^{g}/x^{\ast}M^{g} \longrightarrow 0\] \\
  induces the long exact sequence  \\
   \begin{equation}\label{reg-1}
  \cdots \longrightarrow \Tor_{i}^{R^{g}}(x^{\ast}M^{g},k)_{j}\longrightarrow \Tor_{i}^{R^{g}}(M^{g},k)_{j} \longrightarrow \Tor_{i}^{R^{g}}(M^{g}/x^{\ast}M^{g},k)_{j}\longrightarrow \ \cdots
  \end{equation}
  for all $j\in \mathbb{Z}.$\\
 Since
 \[x^{\ast}[(x^{\ast}M^{g})(1)]= 0 \qquad and \qquad x^{\ast}(M^{g}/x^{\ast}M^{g})=0,\]\\
 $M^{g}/x^{\ast}M^{g}$ is Koszul and $x^{\ast}M^{g}$ has 1-linear resolution by $(1)$.
 Therefore, \ref{reg-1} implies that $\reg_{R}(M)\leq 1$.\\
 \indent $(3)$ Let $\m^{\ast}$ be the maximal homogeneous ideal of $R^{g}$. We have  $(\m M)^{g}=(\m^{\ast}M^{g})(1)$.
 Thus $\m M$ is Koszul if and only if the graded module $\m^{\ast}M^{g}$ has  $1$-linear resolution.
 Set $t=\dim_{k}M/\m M$.  Then, the exact sequence

 \[ 0\longrightarrow (\m^{\ast}M^{g})(1) \longrightarrow M^{g}(1) \longrightarrow k(1)^{t} \longrightarrow 0\]
 induces the long exact sequence
 \begin{equation*}\label{reg-2}
 \cdots \longrightarrow \Tor^{R^{g}}_{i}(k^{t} ,k)_{j+1} \longrightarrow \Tor_{i}^{R^{g}}((m^{\ast}M^{g})(1),k)_{j}\longrightarrow \Tor_{i}^{R^{g}}(M^{g},k)_{j+1} \longrightarrow \ \cdots  .
  \end{equation*}
  As $R$ is Koszul and $\reg_{R}(M)\leq 1$, we obtain, $\Tor_{i}^{A}(\m^{\ast}M^{g}(1),k)_{j}=0$ for $j> i. $
This implies that $\m M$ is Koszul.\\
  \indent $(4)$ Let $I\subseteq \m$ be an ideal of $R$ and we denote by  $"-"$ the natural image of the homomorphism $R\rightarrow R/I$.
  We consider two cases. If $\bar{x}\neq 0$, then $\bar{x}$ is a
   Conca generator of $\bar{\m}$ and therefore $\bar{R}$ is a Koszul ring. If $\bar{x}=0$,
 then there exists an ideal $J$ of $R$ such that
  $I=xR +J$ and $x\notin J$. Hence, $x+J\in R/J$ is a Conca generator of the maximal ideal of $R/J$. This implies $R/J$ is Koszul and thus, $k$ has  a linear resolution as an $(R/J)^{g}$-module.
  Also, one can see that
  $I^{*}=x^{*}R^{g}+J^{*}$, where $I^{*}$ and $J^{*}$ are ideals generated by initial forms of elements of $I$ and $J$ , respectively.
  Now, since $(R/J)^{g}= R^{g}/J^{*}$ and $(R/I)^{g}=R^{g}/I^{*}$, the conclusion follows from \ref{Conca-g}, in conjunction with the fact that $k$ has  $0$-linear resolution over $(R/J)^{g}$.
 \end{proof}

For a Cohen-Macaulay local ring $R$, Abhyankar (\cite{A}) proved  the following inequality
\[e\geq h+1,\]
where $e$ is the multiplicity of $R$ and $h=\textrm{embdim}R-\dim R$. We say that a Cohen-Macaulay local ring
$R$ has minimal multiplicity if the equality $e= h+1$ holds.

The following result extends to modules a result that was already known for local rings (see \cite[Theorem 5.2]{H-V-W}).

\begin{prop}\label{min-mult}
Let $(R,\m,k)$ be a Cohen-Macaulay local ring of dimension $d>0$ and of  minimal multiplicity. Assume that $k$ is infinite.
Then any $R$-module annihilated by a minimal reduction of $\m$ is a Koszul module. In particular, $R$ is a Koszul ring.
\end{prop}
\begin{proof}
Let $J$ be a minimal reduction of $\m$. Then, there exist  a maximal superficial
sequence $a_{1},\cdots,a_{d}$  such that $J=(a_{1},\cdots,a_{d})$ (see \cite{H-Sw}).
For all $i=1,\cdots,d$ set $a^*_{i}$ be the corresponding  initial form of $a_{i}$ in $R^{g}$. Also, set $J^{*}$ be
 the graded ideal generated by the initial forms of $J$. Then, by \cite[Corollary 2.6]{RV-Lect}, $R^{g}$ is a Cohen-Macaulay ring and the sequence
$ a^{*}_{1},\cdots,a^{*}_{d}$ is an $R^{g}$-regular sequence. Now, setting $n:=\dim_{k}R^{g}_{1}$,
 there exist $b_{d+1},\cdots,b_{n}$ in $R^{g}_{1}$ such that \\
\[R^{g}_{1}=\langle a^{*}_{1},\cdots,a^{*}_{d},b_{d+1},\cdots,b_{n}\rangle_{k} .\]\\
  Since $R$ has minimal multiplicity, again by \cite[Corollary 2.6]{RV-Lect}, we have
  \[Jm^i=m^{i+1} \qquad \text{for all}\ i\geq1.\] So,
   \[J^{*}=(a^{*}_{1},\cdots,a^{*}_{d}), \qquad   J^{*}(R^{g}_{1})=(R^{g}_{2})\]
   and then \[b_{i}R^{g}_{1}\subseteq (a^{*}_{1},\cdots,a^{*}_{d})(R^{g}_{1}) \qquad \text{for}\ d+1\leq i\leq n .\]
   Thus, for  $d+1\leq i\leq n$ we have
  \[((a^{*}_{1},\cdots,a^{*}_{d},b_{d+1},\cdots,b_{i-1}) :b_{i}) = (R^{g}_{1}).\] Also,
  \[(0:a^{*}_{1})=0\quad \text{and}\quad (a^{*}_{1},\cdots,a^{*}_{i-1}:a^{*}_{i})=(a^{*}_{1},\cdots,a^{*}_{i-1}),\]
  for all $i=1,\cdots d$.
   which implies that the  ordered basis $\{a^{*}_{1},\cdots,a^{*}_{d},b_{d+1},\cdots,b_{n}\}$ of $R^{g}$ is a flag basis for $R^{g}$.\\
  Now, one has \[R^{g}/J^{*} \cong (R/J)^{g}\] and
 \[ [(R/J)^{g}]_{i}=0 \qquad \text{for all}\ i\geq 2.\]
 Therefore, if $M$ is an $R$-module such that $JM=0$, then $J^{*}M^{g}=0$
and  it is clear that $M^{g}$ has  $0$-linear resolution over $R^{g}/J^{*}$.
 Hence the result follows from \ref{EQ-Koszul}.

\end{proof}

\smallskip

Fitzgerald established in \cite[Theeorem 3.6]{F} that a local ring $(R,\m) $ with $\m^{3}=0$ and $\ann(x)\m=\m^2$, for all $x\in \m\setminus \m^2$,
is Koszul. But the Koszul property of modules over this class of rings is not known.

Below, we show that
the associated graded ring of a local ring which satisfies the above condition, admits a Koszul
filtration. And we give an alternative proof for the result  obtained by Fitzgerald.
 Furthermore, we will give information on the  Koszulness and the regularity of modules over the local rings with this property.

\begin{thm}\label{Fitz}
Let $(R,\m)$ be a local ring and $M$ be an $R$-module. If $\ann(x)\m=\m^2$, for all $x\in \m\setminus \m^2$, then the following hold.
\begin{enumerate}
\item[$(i)$] $M$ is Koszul if there exists $x\in \m\setminus \m^{2}$ such that $\ann(x)M$=0;\\
In particular, $R$ is a Koszul ring.
\item[$(ii)$] if $x_{1},\cdots,x_{s}$ are in $\m\setminus \m^{2}$, then $(x_{1},\cdots,x_{s})M$ is Koszul;
\item[$(iii)$] $\reg _{R}(M)\leq 1$;
\item[$(iv)$] for every ideal $I \subseteq \m$ the quotient ring $R/I$ is Koszul;
 \end{enumerate}
\end{thm}
\begin{proof}
Let $\m$ be minimally generated by $x_{1},\cdots,x_{r}.$ Then the assumption implies that
$\m^{2}\subseteq \bigcap_{i=1} \ann(x_{i})$. Hence $\m^{3}\subseteq \m \bigcap_{i=1} \ann(x_{i})$ $=0$.\\
 \indent   If $x\in \m\setminus \m^2$ and  $r\in \ann(x)$ with $r\in \m\setminus \m^{2}$, then  \[(r+\m^{2})(x+\m^2)=rx+\m^{3} = 0.\]
 Now, clearly \[\ann(x)^* = \ann(x^{*})\qquad \forall x\in \m\setminus \m^2, \]
 where $\ann(x)^{*}$ is the ideal of $R^{g}$ generated by the initial forms of elements
 of $\ann(x)$ and $x^{*}$ is the initial form of $x$ in $R^{g}$.
Let $r$ be an arbitrary element in $ \m^2$. Then, by the assumption, $r\in \ann(x)\m$ and
we may write \[r=\sum r_{i}z_{i} \ \ \text{where}\ r_{i}\in \m\setminus \m^2 ,z_{i}\in \ann(x)\setminus \m^2.\]
This implies
\begin{equation}\label{2}
R^{g}_{2}\subseteq ( \ann(x)^{*}R^{g}_{1}) =( \ann(x^{*})R^{g}_{1}).
\end{equation}
Therefore, for  each non-zero element  $l\in R^{g}_{1}$,  $R^{g}_{2}\subseteq \ann(l)$.
This shows that $\ann(l)$ is generated by linear forms. \\
  \indent Now, we claim that the family of ideals
   \[ \mathcal{F}:=\{I \subseteq R^{g}: I  \ \text{is  generated by  linear forms}\}\]\\
   is a Koszul filtration of $R^{g}$. Set $\m^{*}$ be the maximal homogeneous ideal of $R^{g}$.
   Clearly $(0)$ and $\m^{*}$ are in $\mathcal{F}$.
  Let $I$ be a non-zero element of $\mathcal{F}$ generated by the linear forms
  $l_{1},\cdots,l_{i-1},l_{i}$ with $l_{i}\neq0$. Set $J=(l_{1},\cdots,l_{i-1}),$
  obviously $I/J$ is cyclic. We have
  \[R^{g}_{2}\subseteq \ann(l_{i})\subseteq ((l_{1},\cdots,l_{i-1}):l_{i})\subseteq(J:I).\]
   Hence, $(J:I)\in\mathcal{F}$.
 \vskip 2mm
   In order to  prove $(i)$, consider $x\in \m\setminus \m^{2}$. By (\ref{2}), \
   $ \ann_{R}(x)^{*}$ is generated by linear forms and so
\[\ann(x)^*\in \mathcal{F}.\]

Let $\ann(x)^{*}=(l_{1},\cdots,l_{r},)$ for suitable independent linear forms
 $l_{i}\in R^{g}_{1}$. We can extend $l_{1},\cdots,l_{r}$ to a basis
 $\{l_{1},\cdots,l_{r},l_{r+1},\cdots,l_{n}\}$ of $R^{g}_{1}$, where $n=\dim_{k}(R^{g}_{1})$.
  Let $1\leq i\leq n-r$. Set \[I_{r+i}= (l_{1},\cdots,l_{r+i}) .\] Then we have\\
   \[I_{r}\subseteq I_{r+1}\subseteq \cdots \subseteq I_{n}= \m^{*}\]
 and one can see that \[(I_{j-1}:I_{j})= I_{n}\qquad  for \ all \ r+1\leq j\leq n.\]
 Hence $I_{r}\subseteq I_{r+1}\subseteq \cdots \subseteq I_{n}= \m^{*}$ is a flag in $\mathcal{F}$ starting from $I_{r}$.
 Now, if  $N$ is a graded $R^{g}/I_{i}$-module
for some $r\leq i\leq n$, then,  by Proposition \ref{main}, we have \\

\begin{equation}\label{3}
\p^{R^{g}}_{N}(s,t) = \p^{R^{g}/I_{i}}_{N}(s,t) \p^{R^{g}}_{R^{g}/I_{i}}(s,t).
 \end{equation}\\
 Let $N$ be a graded $R^{g}/I_{r}$-module
  generated by elements of degree $d$. Then,  since $(R^{g}/I_{r})_{i}=0$ for $i\geq 2 $,
  $N(d)$ has $0$-linear resolution over $R^{g}/I_{r}$.
  Also, since $I_{r}\in\mathcal{F}$, the $R^{g}$-module  $R^{g}/I_{r}$
  has  $0$-linear resolution, by Proposition \ref{fit-lin}. Now, using \ref{3} and Remark \ref{nota}(6) we conclude
  that $N$ has $d$-linear resolution over $R^{g}$.\\
 \indent Let $x\in \m\setminus \m^2$ such that $\ann(x)M=0$.
 This implies that $\ann(x)^{*}M^{g}=0.$ Hence $M^{g}$ is an $R^{g}/\ann(x)^{*}$-module
  and by the above argument  $M^{g}$ has  $0$-linear resolution  over $R^{g}$.
   Therefore, by Theorem \ref{Koszul-eq}, $M$ is Koszul . \\
 \indent $(ii)$ Let  $x_{1},\cdots,x_{s}\in \m\setminus \m^{2}$. For all $i=1,\cdots,s$ set $x_{i}^{*}$
be the initial form of $x_{i}$ in $R^{g}$. First we show, by induction on $s$, that for every  graded $R^{g}$-module $N$
 generated by elements of degree $0$, the module  $(x^{*}_{1},\cdots ,x^{*}_{s})N$ has
 $1$-linear resolution over $R^{g}$.\\
  \indent Let $s=1$. Since
 $\ann(x^{*}_{1})(x^{*}_{1}N)=0$, by the proof of $(i)$, we obtain that
 $x^{*}_{1}N$ has  $1$-linear resolution as an $R^{g}$-module.
 Now, let $s> 1$ and suppose that the result has been proved for smaller values of $s$.
  Since
 \[(x^{*}_{1},\cdots ,x^{*}_{s})N/(x^{*}_{1},\cdots ,x^{*}_{s-1})N=x^{*}_{s} [N/(x^{*}_{1},\cdots ,x^{*}_{s-1})N],\]
 by the inductive hypothesis, the $R^{g}$-module  $(x^{*}_{1},\cdots ,x^{*}_{s})N/(x^{*}_{1},\cdots ,x^{*}_{s-1})N$
has  $1$-linear resolution. Now, using the exact sequence\\
 \[0\longrightarrow(x^{*}_{1},\cdots ,x^{*}_{s-1})N\longrightarrow (x^{*}_{1},\cdots ,x^{*}_{s})N\longrightarrow (x^{*}_{1},\cdots ,x^{*}_{s})N/(x^{*}_{1},\cdots ,x^{*}_{s-1})N\longrightarrow 0\]\\
 in conjunction with inductive hypothesis, one can see that  $(x^{*}_{1},\cdots ,x^{*}_{s})N$ has $1$-linear resolution.\\
 Now, set $\mathbf{x}:=(x_{1},\cdots,x_{s})$ and $\mathbf{x}^{*}:=(x^{*}_{1},\cdots,x^{*}_{s})$. Then there is an exact sequence\\
 \begin{equation}\label{4}
0\longrightarrow L(-1) \longrightarrow [\mathbf{x}M]^{g}(-1) \xrightarrow{{\rho}}\mathbf{x}^{*}M^{g}\longrightarrow 0
\end{equation}
 where $L=\mathbf{x}M\bigcap \m^{2}M/\mathbf{x}\m M$ and $\rho$ is the natural surjective homomorphism.
This  yields the long exact  sequence\\
 \begin{equation}\label{5}
 \cdots \rightarrow \Tor_{i}^{R^{g}}(L,k)_{j-1}\rightarrow \Tor_{i}^{R^{g}}([\mathbf{x}M]^{g},k)_{j-1} \rightarrow
 \Tor_{i}^{R^{g}}(\mathbf{x}^{*}M^{g},k)_{j}\rightarrow  \cdots.
 \end{equation}

  Since $R^{g}$ is Koszul, $\Tor_{i}^{R^{g}}(L,k)_{j-1}=0$ for $i> 0$ and  $j\neq i+1.$
 As we already showed, $\mathbf{x}^{*}M^{g}$ has $1$-linear resolution.
 Now, using (\ref{5}), $\mathbf{x}M$ is Koszul.\\
  $(iii)$  We have the exact sequence\\

 \[0\longrightarrow \m^{*}M^{g}\longrightarrow M^{g}\longrightarrow M^{g}/\m^{*}M^{g}\longrightarrow 0\]\\
 of $R^{g}$-modules.
 By $(ii)$, $\reg_{R^{g}}(\m^{*}M^{g})=1$. So from the above exact sequence we get\\
 \[\reg_{R^{g}}(M^{g})\leq max \{\reg_{R^{g}}(\m^{*}M^{g}),\reg_{R^{g}}(M^{g}/\m^{*}M^{g})\}=1.\]\\

 $(iv)$ Let $I\subseteq \m$ be an ideal of $R$. Set $\bar{R}=R/I$ and $\bar{\m}=\m/I$. Then one can see that $\ann_{\bar{R}}(y)\bar{\m}={\bar{\m}}^{2}$ 
 for every $y\in\bar{\m}\setminus {\bar{\m}}^{2}$. Thus $\bar{R}$ is a Koszul ring.

\end{proof}

\begin{rem}
Let $(R,\m)$ be a local ring  with $\m^{3}=0$ and $\mu(\m^{2})=1$. Then, by \cite[Proposition 4.1]{F}, the condition
$\ann(x)\m=\m^2$, for all $x\in \m\setminus \m^2$, is equivalent to  Koszulness of $R$.
Hence, by \cite[Theorems 4.1 and 1.1]{AIS}, every $R$-module has a Koszul syzygy module.\\
\indent Without considering the restriction
$\mu(\m^{2})=1$, Theorem \ref{Fitz} gives us some class of modules which are Koszul.
\end{rem}
Now, it is natural to ask:

\textbf{Question}: Let $(R,\m, k)$ be a local ring such that $\ann(x)\m=\m^2$ for each $x\in \m\setminus \m^2$.
Does  each $R$-module $M$ has a Koszul syzygy module?

 \vskip 1 cm

\textbf{Acknowledgments}\\

This work was done while the author visiting University of Genoa, Italy. He is thankful to the university for the
hospitality. He would like to give special thanks to Professor Maria Evelina Rossi for her guidance and valuable comments.
The author also wish to thank A. Conca, M. Jahangiri and M. Mandal  for many stimulating discussions.\\
\indent The author also would like to thank the referee for his/her valuable comments.

\end{document}